\documentclass[twoside, a4paper, nofootinbib,10pt]{article}
\usepackage{amssymb}
\usepackage[mathscr]{eucal}
\usepackage[dvips]{graphicx}
\usepackage{amsmath}
\usepackage{amsthm}
\def \ni{\noindent}

\newcommand{\be}{\begin{equation}}
\newcommand{\ee}{\end{equation}}
\newcommand{\ben}{\begin{equation*}}
\newcommand{\een}{\end{equation*}}
\newcommand{\bes}{\begin{eqnarray}}
\newcommand{\ees}{\end{eqnarray}}
\newcommand{\besn}{\begin{eqnarray*}}
\newcommand{\eesn}{\end{eqnarray*}}
\renewcommand{\thefootnote}{\alph{footnote}}	
\newcommand{\txt}{\textrm}

\newtheorem*{theorem*}{Theorem}
\newtheorem*{corollary*}{Corollary}
\newtheorem{lemma}{Lemma}
\newtheorem{definition}{Definition}

\title{Big Free Groups are Almost Free}
\author{Tamer Tlas} 
\date{}

\begin{document}
\maketitle

\begin{abstract}
\ni
\textit{It is shown that the big free group (the set of countably-long words over a countable alphabet) is almost free, in the sense that any function from the alphabet to a compact topological group factors through a homomorphism. This statement is in fact a simple corollary of the more general result proven below on the extendability of homomorphisms from subgroups (of a certain kind) of the big free group to a compact topological group.\\ \\
\let \thefootnote\relax} 
\end{abstract}

It is an elementary fact that the free group over a set $A$, $F(A)$ can be defined in two equivalent ways:

\begin{definition}
$F(A)$ is the set of all finite, reduced words over the alphabet $A$.
\end{definition}

\begin{definition}
\label{def:free}
$F(A)$ is the unique, up to an isomorphism, group such that any function $f : A \to G$, where $G$ is some group, factors through a homomorphism from $F(A)$ to $G$.
\end{definition}

The first definition suggests a natural generalization of the concept of a free group: what happens if the finiteness requirement on the words is dropped? Indeed, the study of such generalizations can be traced as far back as \cite{higman}. Recently, such groups have been under intense study as it was realized that, in addition to their intrinsic interest, they play an important role in the study of the fundamental groups of spaces which are not semilocally simply-connected \cite{desmit, cannon}. Such groups also appear in the study of smooth loop groups \cite{tlas}, and thus are relevant for the theory of gauge connections on principal bundles. \\

Let us give now the precise definitions of the group of  transfinite words. We follow \cite{cannon} closely, to which the reader is referred to for more details if needed.

\begin{definition}
Let $A$ be the alphabet set and let $A^{-1}$ be the set of formal inverses of elements of $A$. A transfinite word is a map $w$ from a countable, linearly ordered set $S$ into $A \cup A^{-1}$ such that the preimage of any element of $A \cup A^{-1}$ is finite. 
\end{definition}

Intuitively, a transfinite word is a countable string of letters such that each letter appears at most finitely many times. Two words $w_1 : S_1 \to A \cup A^{-1}$ and $w_2 : S_2 \to A \cup A^{-1}$ are considered identical if there is a bijection $f : S_1 \to S_2$ such that $w_1 = w_2 \circ f$. In this paper we will only deal with words for which both $A$ and $S$ are countable. We thus take $A = \{a_1, a_2, \dots \}$.\\

Transfinite words can be multiplied in essentially the same way as the finite ones:

\begin{definition}
If $w_1 : S_1 \to A \cup A^{-1}$ and $w_2 : S_2 \to A \cup A^{-1}$ are two words, then $w_1 w_2 : S_1S_2 \to A \cup A^{-1}$ is the transfinite word which acts in the obvious way on the domain $S_1S_2$ consisting of the disjoint union of the elements of $S_1$ with $S_2$ with all elements of $S_1$ preceding those of $S_2$. 
\end{definition}

Reduction is a little more involved to formulate in the transfinite case:

\begin{definition}
Denoting $\{s \in S : a \leq s \leq b \}$ by $[a,b]_S$, we say that the word $w : S \to A \cup A^{-1}$ admits a cancellation if there is a subset $T$ of $S$ and a mapping $\ast : T \to T$ satisfying the following four conditions for all $t \in T$:
\begin{itemize}
\item $\ast$ is an involution.
\item $[t, t^\ast]_S = [t, t^\ast]_T$.
\item $[t,t^\ast]_T = ([t, t^{\ast}]_T)^\ast$.
\item $w(t^\ast) = w(t)^{-1}$.
\end{itemize}
Denoting $S - T$ by $S/\ast$ and $w$ restricted to $S/\ast$ by $w/\ast$ we say that $w/\ast$ arises by a cancellation from $w$. If a word does not admit cancellations, it will be called reduced.
\end{definition}

We shall consider all the words which are related to each other by a cancellation to be equivalent. It can be shown that any such equivalence class contains a unique reduced word and that the set of such equivalence classes becomes a group \cite{cannon}, which is called the big free group over $A$, denoted by $BF(A)$.\\

The big free group is known to be not free, and there is quite some work on its free subgroups \cite{desmit, cannon, eda}. In this paper, we show that the big free group is almost free in the sense that it almost satisfies definition \ref{def:free}. We are going to show that any function from $A$ to a compact topological\footnote{As is customary, we assume that $G$ is Hausdorff.} group $G$ can be factored through a homomorphism from $BF(A)$ to $G$, with the only difference being that this homomorphism is not unique. In fact, we are going to show more: Any homomorphism from a subgroup of $BF(A)$, of a special kind, to a compact topological group $G$ can be extended to a homomorphism from the whole of $BF(A)$ to $G$, which will make the factorizability of a function from $A$ a simple special case. Let us define the special class of subgroups of $BF(A)$ that we need:

\begin{definition}
A subgroup $H$ of $BF(A)$ is called tame if for any reduced $w \in H$ we have that every subword of $w$ is also in $H$, where a subword is the restriction of the word $w: S \to A \cup A^{-1}$ to a set of the form $[a,b]_S$.
\end{definition}

We can now state the main result:

\begin{theorem*}
Let $H$ be a tame subgroup of $BF(A)$ and $f : H \to G$ a homomorphism where $G$ is a compact topological group, then $f$ extends to a homomorphism from $BF(A)$ to $G$. 
\end{theorem*}

The main idea of the proof of the theorem is that one `excises' small intervals around singular (to be defined below) points, thus replacing the initial word with a finite string of elements from the tame subgroup. This string is then mapped to $G$ by applying $f$ to the product of the elements in this string. After this, one essentially `takes the limit' as the lengths of these removed intervals go to zero.\\ 

Let us begin with the first lemma: 

\begin{lemma}
Let $\mathcal{I}$ be a directed set and let $\mathcal{G}$ be the group (under pointwise multiplication) of all functions from $\mathcal{I}$ to a compact, topological group $G$ (i.e. the group of all nets in $G$ indexed by $\mathcal{I}$). Let $\mathcal{G}_0$ be the subgroup of $\mathcal{G}$ consisting of those nets which are eventually constant. If $\pi_\mathcal{I}: \mathcal{G}_0 \to G$ is the natural homomorphism given by $\pi_\mathcal{I}(g_\alpha) = \lim_\alpha g_\alpha$ then $\pi_{\mathcal{I}}$ has an extension to all of $\mathcal{G}$.
\end{lemma}

\begin{proof}

Choose an ultrafilter on $\mathcal{I}$ and let $\ast G$ stand for the set of equivalence classes of $G$-valued nets where we consider two nets equivalent if they are identical on an element of the ultrafilter. Proceeding as is customary in nonstandard analysis \cite{goldblatt, encyclo}, it is easy to see that $\ast G$ is a group. Since it is compact, any element of $\ast G$ is near standard, which allows us to define the standard part map from $\ast G \to G$. This map is a homomorphism and is an extension of $\pi_\mathcal{I}$. From here onwards, $\pi_\mathcal{I}$ denotes this extension. 
\end{proof}

Note that the extension above is not unique since it depends on the ultrafilter\footnote{For a concrete example of this nonuniqueness take $G = U(1)$ and $\mathcal{I} = \mathbb{N}$. Then if the odd naturals are in our ultrafilter, $\pi_{\mathcal{I}}$ will assign the value $-1$ to the net $g_n = (-1)^n$. While, if the ultrafilter contained the even naturals, this same net will be assigned the value of $+1$.}.\\

It suffices to prove the Theorem in the case when the tame subgroup $H$ contains all the letters of the alphabet $A$. This follows at once from:

\begin{lemma}
If $A' \subset A$, then any homomorphism $f$ from $BF(A')$ to a compact group $G$ extends to $BF(A)$.
\end{lemma}

\begin{proof}
This is an immediate consequence of the fact that there is a retraction from $BF(A)$ to $BF(A')$ (obtained simply by deleting the letters not contained in $A'$)\footnote{We would like to thank the referee for pointing this out.}. However, let us give an alternative proof since it illustrates in a simple context an idea which will be used later. \\

Consider the set of all finite collections of words in $BF(A')$. Order this set by inclusion making it into a directed set $\mathcal{J}$. Suppose $w \in BF(A)$. By deleting the letters not appearing in $A'$, this word splits into a string of elements in $BF(A')$. Let us denote this string by $w'$ (note that $w'$ is \textit{not} a single element of $BF(A')$, rather it is a string of elements of $BF(A')$). For example, if $A = \{a, b, c\}$, $A' = \{a, b\}$ and $w = a^2ba^{-1}cb^2c^2a^3$, then $w'$ is the string of three words $(a^2ba^{-1}) (b^2) (a^3)$. \\

Now pick any element $j \in \mathcal{J}$ and associate to it the group element $f(w_1) f(w_2) \dots $ where $w_1, w_2, \dots$ are the elements of $j$ appearing in the string $w'$ in the order in which they appear in it (the appearance of an inverse of an element of $j$ is counted as an appearance of the element). Note that since any given letter can appear only finitely many times, the product is a finite one. If no elements of $j$ appears in the string $w'$, associate to $j$ the identity element. We thus get a $G$-valued net indexed by $\mathcal{J}$ which we shall denote by $\{ w_j \}_{j \in \mathcal{J}}$. Note that if $w \in BF(A)$ happens to be in $BF(A')$, then this net is eventually constant and is equal to $f(w)$. It is easy to check that if $w_1$ and $w_2$ are any two elements in $BF(A)$ then eventually $ (w_1 \cdot w_2)_j = (w_1)_j (w_2)_j$. This is because the only case when $ (w_1 \cdot w_2)_j = (w_1)_j (w_2)_j$ may fail to hold is when $w_1 = a b c d$ and $w_2 = d^{-1} c^{-1} e f$ where $b, c, e \in BF(A')$ and the word $be$ is irreducible. However, if we let $j_0 = \{b, c, e, bc , c^{-1} e, b e \}$ then for any $j$ containing $j_0$ we do have the equality we want. It follows at once that the map $w \to \{ w_j \}_{j \in \mathcal{J}} \stackrel{\pi_\mathcal{J}}{\to} G$ is a homomorphism which extends $f$. \end{proof}

In view of the above discussion, we shall assume below that $H$ includes all the letters of $A$.\\

Let now $\mathcal{I}$ be the following set

\ben
\mathcal{I} = \bigg \{ \{l_n\}_{n=1}^\infty \quad : \quad  l_{n+1} \leq l_n  \quad , \quad l_n > 0 \quad , \quad \sum_{n=1}^\infty l_n < \infty  \bigg \},
\een

i.e. the set of all monotone non-increasing, strictly positive, real-valued sequences whose sum is convergent. Order this set by stipulating that $\{l_n \}_{n=1}^\infty \prec \{ l'_n\}_{n=1}^\infty \iff l_n \geq l'_n \, , \, \forall n$. It is obvious that $(\mathcal{I}, \prec)$ becomes a directed set. \\

Let $w \in BF(A)$ be a reduced word. Denote by $k_n$ the total number of times the letter $a_n$ appears in $w$, where we count both the letter and its inverse. Suppose that $\iota \in \mathcal{I}$ is such that $\sum_{n=1}^\infty k_n l_n \equiv L_w < \infty$. We decompose $[0,L_w]$ into two sets, a countable collection of disjoint open intervals and its complement, where the intervals correspond to the letters of the word $w$. This is done in the following way: For every letter in the word, associate an open interval of length $l_n$, if the letter is $a_n$ or $a_n^{-1}$, whose starting point is equal to the sum of the lengths of the intervals corresponding to all the letters preceding the given letter. Thus, if for example our word is $a_2 a_1^2 a_2^{-1}$, we get the following intervals $\{(0,l_2) , (l_2,l_2+l_1), (l_2+l_1, l_2 + 2 l_1), (l_2 + 2l_1, 2l_2 + 2 l_1)  \}$. It is obvious in this example and, as is easy to check, true generally, that any two such obtained intervals are disjoint and that each one is a subset of $[0, L_w]$. Denote the complement (in $[0,L_w]$) of the union of these intervals by $C$. $C$ is clearly a closed set. There is an obvious correspondence between subwords of $w$ and subintervals of $[0,L_w]$ with endpoints\footnote{It is irrelevant whether the endpoints are included or not in these subintervals as no letters correspond to endpoints.} in $C$. \\

Let $x \in C$. Note that any such point naturally splits $w$ into the part `before $x$' and the part `after $x$'. We now make:  

\begin{definition}
We say that $x$ is \textit{regular on the right/left} if there is an initial/final segment of the word after/before $x$ which is contained in $H$. If a point is both regular on the left and on the right, then we shall simply say that it is \textit{regular}. Points which are not regular will be called \textit{singular} (note that a singular point can be regular on the right or on the left). The set of all singular points will be denoted by $C'$.
\end{definition}

\begin{lemma}
$C'$ is closed.
\end{lemma}

\begin{proof}
This is an immediate consequence of the fact that $H$ is a tame subgroup. To see this, pick any regular point. By definition this means that there is an open interval around it (with endpoints in $C$) such that the word corresponding to this interval is in $H$. This implies that any element of $C$ in this interval is also regular for there is an interval around it whose word is contained in $H$ (and all subintervals are also in $H$). Thus the set of regular points is open (in C), which means that $C'$ is closed.
\end{proof}

Fix now $m \in \mathbb{N}$. To every $x \in C'$ we associate two intervals, $I_{x.m}^r$ and $I_{x,m}^l$ in the following way:

\begin{itemize}
\item If $x$ is regular on the right, then $I_{x,m}^r = [x,x] = \{x \}$.
\item If $x$ is not regular on the right, let $\alpha_{x,m} = \txt{sup} \{ x \in C' \cap [x, x + \frac{1}{m} ] \}$. We now have two subcases:
\begin{itemize}
\item $\alpha_{x,m} = x$. In this case $I_{x,m}^r = [x, x + \frac{1}{m}] \cap [0, L_w]$.
\item $\alpha_{x,m} \neq x$. In this case $I_{x,m}^r = [x , \alpha_{x,m}]$.
\end{itemize}
\item $I_{x,m}^l$ is defined with obvious changes in an analogous way.
\end{itemize}

Consider now the set $C_m$ given by: $$C_m = \bigcup_{x \in C'} (I_{x,m}^r \cup I_{x,m}^l). $$ Note that $C' \subset C_m$ and that any connected component of $C_m$, being a subset of $\mathbb{R}$, is an interval.\\

Now that we know that the connected components of $C_m$ are all intervals, we shall classify them into two classes. The first class are those whose length is greater or equal to $\frac{1}{m}$, while the second one are those whose length is strictly less. We have the following lemma:

\begin{lemma}
If $I$ is an interval of the second class, then its endpoints are in $C$. Additionally, unless one of the points involved is an endpoint of $[0,L_w]$, the distance between the two left endpoints of any two intervals of the second class is always greater or equal to $\frac{1}{m}$ with the same being true for right endpoints.
\end{lemma}

\begin{proof}
We will use the following basic fact from point-set topology:\\

\textit{If an interval is a union of a collection of intervals, then the left endpoint of the original interval is contained in the closure of the left endpoints of the intervals in the collection, with the same being true for right endpoints.}\\

We know that $I$ is a union of intervals of the form $I_{x,m}^r$ and $I_{x,m}^l$ where $m$ is held fixed and $x$ ranges of a subset of $C'$. It is clear that the right endpoint of any $I_{x,m}^l$, which is just $x$, is in $C'$ and thus in $C$. On the other hand, the right endpoint of any $I_{x,m}^r$ can be:

\begin{itemize}
\item Equal to $L_w$ and thus is in $C$.
\item Equal to $\alpha_{x,m}$ in which case it is in $C'$, since $C'$ is closed (recall the definition of $\alpha_{x,m}$ above). In this case the right endpoint is also in $C$.
\item Equal to a point not in $C$. 
\end{itemize}

The last case however, can only happen when $\alpha_{x,m} = x$, in which case the length of $I^r_{x,m}$ is equal to $\frac{1}{m}$. Since $I$ is assumed to be an interval of the second class, i.e. its length is strictly less than $\frac{1}{m}$, this case cannot occur. We thus have that the right endpoints of all the intervals whose union is equal to $I$ are all contained in $C$. It follows that the right endpoint of $I$ is in $C$ as well. Needless to say, the same argument shows that the left endpoint of $I$ is in $C$ as well. Note that the discussion above shows that if an endpoint of $I$ is not an endpoint of $[0,L_w]$, then the endpoint is in fact in $C'$. Moreover, since $C' \subset C_m$, then $I$ is closed since its endpoints cannot be in any other connected component of $C_m$. \\

Now suppose we take two intervals of the second class, $[a_1, b_1]$, $[a_2, b_2]$ and assume $a_1 \neq 0$. Since $a_1 \in C'$, we know that it cannot be regular. It has to be regular on the left since otherwise $I_{a_1, m}^l$ being an interval of nonzero length with right endpoint equal to $a_1$ would not be contained in $[a_1, b_1]$. It follows that $a_1$ is not regular on the right. This implies that $a_2 > a_1 + \frac{1}{m}$, for otherwise $[a_1, a_2] \subset I_{a_1,m}^r \subset [a_1, b_1]$ which is a contradiction. Thus the left endpoints of the intervals of second class (if they are not endpoints of $[0,L_w]$) are always at least $\frac{1}{m}$ apart. The same argument shows that the same is true for the right endpoints. This concludes the proof of this lemma. \end{proof}

In view of the above lemma, and since all intervals are a subset of $[0,L_w]$, it is clear that the number of intervals of the second class must be finite. \\

The number of intervals of the first class is also finite. This is because the sum of their lengths (each of which is greater or equal to $\frac{1}{m}$) has to be finite, being bounded from above by $L_w$.\\

It could happen that an interval of the first class has its two endpoints not in $C$. In this case replace it by the smallest closed interval containing it whose endpoints are in $C$.\\

Summarizing the above construction we have obtained, for a fixed word $w$, a choice of $\iota \in \mathcal{I}$ (with the condition that $L_w < \infty$) and for an $m \in \mathbb{N}$, a finite collection of disjoint intervals which contain all the singular points in $[0, L_w]$ and whose endpoints are always in $C$. Now, note that if we delete all the subwords of $w$ corresponding to these intervals\footnote{Recall that there is a correspondence between subwords and subintervals with endpoints in $C$.}, we will be left with a finite string of subwords $w_1 w_2 \dots w_n$. We now state the following:

\begin{lemma}
Each of $w_1, w_2, \dots, w_n$ is in $H$.
\end{lemma}

\begin{proof}
Fix any particular letter $a$ in the subword $w_k$ and consider the set of all words in $H$ which are subwords of $w_k$ containing this particular letter. Every one of such subwords corresponds to a subinterval of the interval corresponding to $w_k$. Let $\alpha$ be the infimum of the left endpoints of these intervals. Similarly, let $\beta$ be the supremum of the right endpoints. We claim that $\alpha$ is regular on the right. This is the case for otherwise $I_{\alpha,m}^r$ would be a nontrivial interval contained in the interval corresponding to $w_k$ which is impossible since all such intervals where deleted. The same argument shows that $\beta$ must be regular on the left.\\

We claim that the word corresponding to $[\alpha, \beta]$ is in $H$. Since $\alpha$ is regular on the right, there is an initial segment of the word corresponding to $[\alpha, \beta]$ such that the word corresponding to it is in $H$. If it was not possible to choose this segment to include the given fixed letter $a$, it would follow that there could be no overlap between the interval corresponding to this initial segment and the interval corresponding to any subword of $w_k$ which is in $H$ and which contains this letter (this is due to the fact that $H$ is a tame subgroup and thus if the intervals corresponding to two words in $H$ overlap, then the word corresponding to the union of the two intervals is also in $H$). However $\alpha$ is the infimum of such intervals and we have a contradiction. The same argument shows that there is a final segment of the word corresponding to $[\alpha, \beta]$ containing the given fixed letter. Again using the fact that $H$ is tame we have that the word corresponding to $[\alpha, \beta]$ is in $H$.\\

If this word was not equal to $w_k$, i.e. if e.g. $\alpha$ was not the left endpoint of $w_k$, it would follow that $\alpha$ is regular and there would be a strictly longer subword than $[\alpha, \beta]$ which would still be in $H$. This would contradict the way $\alpha$ was defined.
\end{proof}

We are now ready to finish:

\begin{proof}[Proof of the Theorem]

The discussion above shows that for any element $\iota \in \mathcal{I}$ (with $L_w < \infty$) and any $m \in \mathbb{N}$ we can associate to $w$ a finite string of words in $H$. Multiplying these words in the order in which they appear in the string gives a word in $H$. Let us denote this word by $h_{\iota, m} (w)$, where we have kept the dependence on $\iota$ and $m$ explicit. It is obvious that if $w \in H$ then $h_{\iota, m} (w) = w$. This is simply because if $w \in H$, there are no singular points and thus no intervals to delete. \\

Let us begin by observing that $h_{\iota, m} (w^{-1}) = ( h_{\iota, m} (w) )^{-1}$. To see this, note that $L_w = L_{w^{-1}}$ as a consequence of the fact that we associate intervals of the same length to a letter and to its inverse. Also, if we have a decomposition of $[0,L_w]$ into intervals corresponding to the letters of $w$, then the decomposition of $[0, L_{w^{-1}}]$ is obtained from it by reflecting through the origin and then shifting to the right by $L_w$ (simply because we exchange the sum of the lengths of the intervals `before' a letter, with those `after' it). Note that if an interval in $[0,L_w]$ corresponded to a letter in $w$, then the reflected and shifted interval in $[0, L_{w^{-1}}]$ corresponds now to the inverse of the letter in $w^{-1}$. \\

Now, it follows from the fact that the definitions of regularity on the right/left are mirror images of each other and from a similar symmetry in the definitions of $I^r_{x,m}, I^l_{x,m}$, that the set $C_m$ we shall delete from $[0, L_{w^{-1}}]$ is obtained from the $C_m$ deleted from $[0,L_w]$ by reflection and translation. This however, means that the finite string of words obtained from $[0,L_{w^{-1}}]$ is the string of words one obtains from $[0,L_w]$ by `reflection', i.e. by rewriting the words in the opposite order, by rewriting the letters in each word in the opposite order and replacing each letter with its inverse (all the intervals are in the opposite order and each interval now corresponds to the inverse of the original letter). It is immediate that the product of the string of words obtained from $[0, L_{w^{-1}}]$ (which is $h_{\iota,m}(w^{-1})$) is the product of the `reflected' words, which is equal to $( h_{\iota, m} (w) )^{-1}$, and we have the equality we want.     \\

Keeping $\iota$ fixed for now, we claim that if $w$ and $\tilde{w}$ are two words such that $L_w, L_{\tilde{w}} < \infty$ then eventually, i.e. for all sufficiently large $m$, we have  
\be
\label{eq:eq}
h_{\iota, m}(w \cdot \tilde{w}) = h_{\iota,m} (w) \cdot h_{\iota,m}(\tilde{w}).
\ee

It is enough to prove this eventual equality for words whose concatenation is irreducible. To see this, assume that this was done. Let $w$ and $\tilde{w}$ be two words and write them as $w=w' \cdot w''$ and $\tilde{w} = w''^{-1} \cdot \tilde{w}'$, where $w''$ is the part that gets reduced when $w$ and $\tilde{w}$ are concatenated. We then have that eventually

\besn
h_{\iota, m} (w \cdot \tilde{w}) & = & h_{\iota, m} (w' \cdot \tilde{w}') \\
& = & h_{\iota, m} (w') \cdot h_{\iota, m} (\tilde{w}') \\
& = & h_{\iota, m} (w') \cdot h_{\iota, m}(w'') \cdot h_{\iota, m}(w''^{-1})  \cdot h_{\iota, m} (\tilde{w}')\\
& = & h_{\iota, m} (w' \cdot w'') \cdot h_{\iota, m} (w''^{-1} \cdot \tilde{w}') \\ 
& = & h_{\iota,m} (w) \cdot h_{\iota,m}(\tilde{w})
\eesn

Therefore assume that $w$ and $\tilde{w}$ are two words whose concatenation is irreducible. How can equality (\ref{eq:eq}) fail to hold? The only way this could happen is when there is a mismatch between the intervals deleted in $w$ and $\tilde{w}$ and the intervals deleted in $w \cdot \tilde{w}$. For instance $\tilde{w}$ could be an element in $H$ (so no subintervals of it should be deleted), while there could be a point in the interval corresponding to $w$ which is not regular on the right and which is sufficiently close to $L_w$ such that the excised interval at this point `spills' over to the interval corresponding to $\tilde{w}$ in $w \cdot \tilde{w}$. This would cause a part of $\tilde{w}$ to be deleted causing (\ref{eq:eq}) to fail.\\

Let us see that this does not happen when $m$ is sufficiently large. There are three cases two consider:

\begin{itemize}
\item $L_w$ is regular on the right and on the left in $[0,L_{w \cdot \tilde{w}}]$: In this case choose $m$ to be large enough so that $L_w$ is more than $\frac{1}{m}$ from the nearest singular point.
\item $L_w$ is regular on the right but not on the left (or vice versa): Let $m$ be large enough so that there are no singular points in $(L_w, L_w +\frac{1}{m}]$. Note that no $I_{x,m}^r$ for $x < L_w$ can have its right endpoint larger than $L_w$ since $\alpha_x \leq L_w \in C'$ for any such $x$. 
\item $L_w$ is not regular on the right nor on the left: In this case any $m$ works. Suppose that $x \in [0, L_w]$ and that $I^r_{x,m}$ `spills over' to $[L_w, L_{w \cdot \tilde{w}}]$, i.e. more precisely $I^r_{x,m}\cap [L_w, L_{w \cdot \tilde{w}}] \neq \phi$ (note that here we are considering $I^r_{x,m}$ for the word $w \cdot \tilde{w}$). However, by assumption $L_w$ is not regular on the right.  We claim that $I^r_{L_w,m} \supset I^r_{x,m}\cap [L_w, L_{w \cdot \tilde{w}}]$. To see this, consider the two cases:
\begin{itemize}
\item $\alpha_{L_w,m} = L_w$: In this case $I^r_{L_w,m} = [L_w, L_w + \frac{1}{m}] \supset \{L_w \} = I^r_{x,m}\cap [L_w, L_{w \cdot \tilde{w}}]$. Note that the last equality holds because $\alpha_{x,m} = L_w$.
\item $\alpha_{L_w,m} \neq L_w$: In this case $I^r_{L_w,m} = [L_w, \alpha_{L_w,m}] \supset [L_w, \alpha_{x,m} ]$. This follows trivially from $\alpha_{L_w,m} \geq \alpha_{x,m}$.
\end{itemize}

Above, we only considered `right' intervals. Needless to say symmetric statements are true regarding $I^l_{L_w,m}$. Thus, the `spillovers' are contained in $I^r_{L_w,m}$ and $I^l_{L_w,m}$. Therefore, in this case, there is no mismatch between the intervals removed from the words whether the words are considered individually or are concatenated.
\end{itemize}

Thus we can always choose $m$ to be large enough so that the intervals deleted from $w$ and from $\tilde{w}$ match those which are deleted from $w \cdot \tilde{w}$. This means that the string of elements of $H$ obtained from $w \cdot \tilde{w}$ is the concatenation of the strings obtained from $w$ and $\tilde{w}$. It follows that (\ref{eq:eq}) holds.\\

Keeping $\iota$ fixed for now, and using $f$ (the given homomorphism from $H$ to $G$), we can associate to $w$ a sequence of elements in $G$, $ m \to g_{\iota, m} =  f( h_{\iota,m}(w)  ) $, where we have kept the dependence on $\iota$ explicit. In view of (\ref{eq:eq}) we have that eventually the sequence that corresponds to  $w \cdot \tilde{w}$ is equal to the sequence $\{g_{\iota, m} h_{\iota, m} \}_{m=1}^\infty$. \\

Using $\pi_\mathbb{N}$ (nets are simply sequences here), we can map the sequence to a single group element $g_\iota$. Let $g_\iota$ be equal to the identity element of $G$ if $L_w = \infty$ (note that for any word $w$, we do have eventually $L_w < \infty$). We thus get for any word a net of group elements $\{ g_\iota \}_{\iota \in \mathcal{I}}$ such that the net corresponding to the product of two words is equal eventually to the pointwise product of the two individual nets. Using $\pi_\mathcal{I}$ again (the nets here are indexed by $\mathcal{I}$ of course), we see that the map $w \to \{ g_\iota \}_{\iota \in \mathcal{I}} \stackrel{\pi_\mathcal{I}}{\to} G$ is the homomorphism extension that we seek.
\end{proof}

We have the following immediate corollary: 

\begin{corollary*}
$BF(A)$ satisfies definition $2$, if $G$ is a compact, topological group.
\end{corollary*}

\begin{proof}
This follows at once from the fact that $F(A)$ is a tame subgroup of $BF(A)$.
\end{proof}

Let us finish by noting that the extension in the Corollary is never unique. To see this let $A = \{a_1, a_2, \dots \}, \alpha = a_1 a_2 \dots$ and $H = F(A \cup \alpha)$ (in other words, $H$ is the free group generated by $A \cup \alpha$). Since any subword of $\alpha$ belongs to $H$, $H$ is tame. The theorem guarantees that any homomorphism from $H$ to a compact $G$ extends to $BF(A)$. However, since $H$ is free, there are infinitely many homomorphisms on it which coincide when restricted to $F(A)$ (they only differ in their action on $\alpha$). Thus no extension of a homomorphism from $F(A)$ to $G$ is unique.\\

\textbf{Acknowledgements:} The author would like to thank an anonymous referee for several suggestions which have greatly improved the manuscript and its readability.

\texttt{{\footnotesize Department of Mathematics, American University of Beirut, Beirut, Lebanon.}
}\\ \texttt{\footnotesize{Email address}} : \textbf{\footnotesize{tamer.tlas@aub.edu.lb}} 

\end{document}